\documentclass{article}
\usepackage{graphicx} 
\usepackage{amsmath,amsthm,amsfonts,amssymb,commath}
\usepackage[utf8]{inputenc}
\usepackage{dsfont}
\usepackage[title]{appendix}
\usepackage{hyperref}
\usepackage{graphicx}
\newtheorem{theorem}{Theorem}
\newtheorem{lemma}[theorem]{Lemma}
\newtheorem{definition}[theorem]{Definition}
\newtheorem{example}[theorem]{Example}
\newtheorem{remark}[theorem]{Remark}
\theoremstyle{remark}

\title{Minimal Spacing of Eigenvalues on Fractals}
\author{Bernard Akwei }
\date{\today}

\begin{document}

\maketitle
\begin{abstract}
\noindent On the unit interval (I), and the Sierpinski Gasket ($\mathcal{SG}$), the spectral decimation function of the Laplacian has similar properties that result in positive minimum spacing of eigenvalues.
Other fractals, for example the level-3 Sierpinski Gasket, $\mathcal{SG}_3$, may not necessarily enjoy these properties.
Our goal is to obtain an easy and sufficient criterion for positive infimum spacing of eigenvalues in the spectrum based on the properties of the spectral decimation function for the appropriate fractal.
We also give a sufficient condition for zero infimum spacing of eigenvalues.
\end{abstract}

\section{Introduction}

Given $\mathds{I}$ is the identity map, the spectrum of an operator, $T$, is the set, $\{\lambda: T-\lambda \mathds{I} \text{ is not invertible}\}$. A complex number 
$\lambda$ is called an eigenvalue of the operator $T$, if  the operator $T-\lambda \mathds{I}$ is not injective. 
We  examine the spacing of the eigenvalues in the spectrum  of the 
self-adjoint non-negative 
Laplacian, $\Delta$, on fractals.
This is closely related to the idea of gaps in the spectrum, which has been  studied   by many authors including the work of  Strichartz in \cite{strichartz2005laplacians}.

Because the spaces we will  consider are compact,  the Laplacian has discrete spectrum and we say gaps occur in the spectrum if, 
\[\limsup _{n\to\infty}\frac{\lambda_{n+1}}{\lambda_n} > 1,\]
where $\{\lambda_n\}$ are eigenvalues \cite{hare2011disconnected}. The study of the spectrum of the Laplacian on Fractals dates back to the works of Rammal and Toulouse in\cite{rammal1983random,rammal1984spectrum}.
Several advances have been made in describing these spectrums. For some fractals, it was discovered that certain properties about their spectrum were different from classical results. The difference were mainly due to the existence of large gaps in the spectrum.
The phenomenon of gaps have interesting consequences.
For example it results in convergence of Fourier series along some subsequence as shown by Strichartz in \cite{strichartz2005laplacians}. As a result, several papers including \cite{hare2011disconnected, drenning2005fractal} have discussed large gaps in the spectrum of the Laplacian.\\

 However, questions about small gaps can be challenging, and very few  articles consider small gaps, as noted by \cite{alonso2022minimal}. Such topics may include questions about spacing of eigenvalues( i.e. the difference between subsequent eigenvalues in the spectrum).
\\~\\
In recent research involving machine learning, the spacing of eigenvalues has been used to obtain useful estimates about relevant data~\cite{armstrong2025optimal}. In \cite{wahl2024kernel}, Wahl provides a non-asymptotic error estimate that results from approximating the eigenspaces of the Laplace-Beltrami operator depending on the spacing of its eigenvalues.
Also from Wielandt's Inequalities, stated in {Theorem}~\ref{wielant_inequalites}, if the infimum spacing of the eigenvalues of A, a symmetric operator on $\mathbb{R}^n$ is positive, then
\begin{equation*}
0 \leq \lambda_j^{\downarrow}(A+B)-\lambda_j^{\downarrow}(A) \leq \frac{\norm{B}^2}{\inf_j\abs{\lambda_j^{\downarrow}(A)-\lambda_{j+1}^{\downarrow}(A)}}, 1 \leq j \leq d
\end{equation*} 
and 
\begin{equation*}
0 \leq \lambda_j^{\downarrow}(A)-\lambda_j^{\downarrow}(A+B)\leq \frac{\norm{B}^2}{\inf_j\abs{\lambda_j^{\downarrow}(A)-\lambda_{j+1}(A)}}, d+1 \leq j \leq n.
\end{equation*}
Thus, we obtain an estimate for the error in the eigenvalues after some perturbation. The extensive research relating to the spacing of eigenvalues motivates our examination of  the infimum spacing.
Often times, it may be technical to obtain the actual value of this infimum.
In the case of $\mathcal{SG}$, \cite{alonso2022minimal} has shown that this coincides with the spectral gap (i.e. the spacing between the first two eigenvalues.) by examining and  considering the dynamics associated with the  spectral decimation function (see { Section}~\ref{SpectralDecimation}). This technique, named spectral decimation in ~\cite{fukushima1992spectral}, played a crucial role in describing the spectrum of $\mathcal{SG}$.
\\~\\
In this paper, we give a sufficient criterion that ensures that the infimum spacing of eigenvalues is positive as well as a sufficient condition for which the infimum spacing must be zero.
\\~\\
In { Section}~\ref{Preliminaries}, we consider some preliminaries including graph approximations and briefly discuss spectral decimation.\\
{Section}~\ref{MinimalSpacing} is the main part of this text. It is divided into two subsections.
In {subsection}~\ref{ZeroSpacing},  we prove {Theorem}~\ref{ZeroInf}, which alludes to the fact that for  positive infimum spacing in the spectrum, the multiplier of the zero fixed point must dominate that of all other fixed points. This is  a partial converse  to {Theorem}~\ref{mainthm} in {subsection}~\ref{PositiveSpacing}, where we show that under certain conditions the infimum spacing is necessarily positive. The main idea is to apply {Lemma}~\ref{main lemma} to  the discrete case and then pass to the continuous case using spectral decimation.\\
In each section, we also show some easy examples that demonstrates the conclusion of the Theorems. Examples ???? are interesting in that they show that certain features that may hold on $\mathcal{SG}$ may not necessarily apply to $\mathcal{SG}_3$

\section{Preliminaries}
\label{Preliminaries}

\subsection{Graph Approximations}
We begin with a discussion of Graph approximation following the notes by Robert Strichartz's from \cite{strichartz2006differential}.

A self similar set $K$,  is defined to be the unique nonempty compact set, satisfying:
\begin{equation}
\label{SelfSimilar}
K=\bigcup_{j=1}^N \varphi_j(K), N\geq 2
\end{equation}
for some distinct injective contraction mappings, $\varphi_j$ on a complete metric space.
Equation~(\ref{SelfSimilar})is known as a self similar identity.

 To  illustrate, we take as an example, the unit interval, $I$.
The mappings,
\begin{equation}
\label{ContractionsI}
F_0(x)=\frac{1}{2}x, F_1(x)=\frac{1}{2}x + \frac{1}{2}
\end{equation}
on I with fixed points $0,1$ respectively, gives the self similar identity,
\begin{equation}
\label{SelfSimilarI}
I=F_0I \cup F_1I
\end{equation}
Of course we could have more iterations by composing the $F_i$'s and writing,
\begin{equation}
\label{}
I=\bigcup_{|w|=m}F_wI, w=(w_1,\dots, w_m), w_j\in \{0,1\}
\end{equation}
we call $w$ a word of length $m=|w|$, $F_wI$ an m-level cell and we write $x\underset{m}\sim y$ if there is an m-level cell having $x,y$ as boundaries.
Let $V_0=\{q_0,q_1\}, q_0=0,q_1=1$ and define inductively
\begin{equation}
\label{vertices}
V_m := \bigcup_iF_iV_{m-1}, i=0,1.
\end{equation}
We regard the sets $V_m$ as the vertices of the graph $\Gamma_m$ with edge relation $x\underset{m}\sim y$
and build the graph $\Gamma_m$ inductively from the graph of $\Gamma_{m-1}$ by taking the two images $F_0\Gamma_{m-1}$ and $F_1\Gamma_{m-1}$.
Observe that $V_m$ is increasing to $V_*$, the set of dyadic rationals in $I$, which is dense in $I$.
And if we are only interested in continuous functions on $I$, then it suffices to know their action on $V_*$.
\\~\\
The case of $\mathcal{SG}$ may be viewed as a generalization of the case of the unit interval ($I$) with
\begin{equation}
\label{ContractionsSG}
F_i(x)=\frac{1}{2}(x-q_i) +q_i, i=0,1,2
\end{equation}
where $q_i$'s are the vertices of a nondegenerate triangle in the plane.
$\mathcal{SG}$ satisfies the self similar identity:
\begin{equation}
\label{SelfSimilarSG}
\mathcal{SG}=\bigcup_{i=0}^2F_i(\mathcal{SG})
\end{equation}
In the case of $\mathcal{SG}$ we take $V_0 = \{q_0,q_1,q_2\}$, the vertices of the triangle.
We construct a graph $\Gamma_m$ with vertices $V_m$ by defining the edge relation $x\underset{m}\sim y$ if there is a cell of level m containing both $x$ and $y$.
$\Gamma_m$ is obtained by taking the three copies of $F_i\Gamma_{m-1}$ of $\Gamma_{m-1}$ and identifying the points:
$\{F_iq_j = F_jq_i\}_{i\neq j}, i,j=0,1,2$.

\subsection{Spectral decimation}
\label{SpectralDecimation}

The precise definition of spectral decimation can be found in~\cite{DGV2012} for the purpose of our work, it suffices to describe it the following way.
\\~\\
 Let $V_m$ denote the vertices of the m-level approximation of $K$.
We define the Discrete laplacian $\Delta_m$ on $V_m$ by:
\begin{equation}
\label{DiscreteLaplacian}
\Delta_mu(x) := \sum_{x\underset{m}\sim y}(u(y)-u(x)), x,y \in V_m.
\end{equation}
\cite{hare2011disconnected} has shown that for a fully symmetric self similar structure on a finitely ramified fractal $K$, there is a function $R$, called the spectral decimation function relating the eigenvalues of $\Delta_m$ and eigenvalues of $\Delta_{m-1}$.
In the paper, \cite[Proposition 2.18]{shima1996eigenvalue}, Shima shows that $R(0)=0$ and $c_\Delta := R'(0)>1$. Thus $0$ is a repulsive fixed point of the spectral decimation.

\begin{definition}\label{stdLapdef}
Define the standard Laplacian on $K$ as 
\begin{equation}
\label{StdLaplacian}
\Delta u(x) = \lim_n c_\Delta^n \Delta_nu(x), u \in C(K)
\end{equation}
if the limit exist.
\end{definition}
Let $\phi_0,\dots, \phi_L$ be the partial inverses of $R$ with $\phi_0$ satisfying, $\phi_0(0)=0$.

If all the eigenvalues of $\Delta$ have the form 
\begin{equation}
\label{eigenvalue}
\lambda_k = \lim_n c_{\Delta}^n\lambda_k^{(n)} = \lim_n c_{\Delta}^{n_0 +n}\phi_0^n\lambda_k^{(n_0)}, k \in \mathbb{N}
\end{equation}
where $\lambda_k^{(n)}$ is the $k-th$ eigenvalue of $\Delta_n$,
then $\Delta$ is said to admit spectral decimation \cite{hare2011disconnected}.

 The class of fully symmetric self similar finitely ramified fractals admit spectral decimation.

 Note that {Definition}~\ref{stdLapdef} describes a phenomenon where a suitable series of eigenvalues in the finite level approximation of a self-similar set is produced by an orbit of a particular dynamical system. This phenomenon is what is termed Spectral decimation~\cite{fukushima1992spectral}. 
\section{Minimal Spacing of Eigenvalues}
\label{MinimalSpacing}
In this section, we 
first consider the condition for zero infimum spacing.
Subsequently, we prove {Lemma}~\ref{main lemma}. 
Since these fractals admit spectral decimation, it suffices to consider what happens in the discrete case and
thus, {Theorem}~\ref{mainthm} will follow from {Lemma}~\ref{main lemma}. 

\subsection{Zero Infimum Spacing}
\label{ZeroSpacing}
The next theorem gives a condition under which the infimum spacing is zero. Here our main result is Theorem~\ref{ZeroInf}, where we show that the existence of a nonzero repulsive fixed point of the spectral decimation function with multiplier, larger than the zero fixed point indicates zero infimum spacing in the spectrum.
\begin{theorem}
\label{ZeroInf}
     Assume $\Delta$ admits spectral decimation with spectral decimation function $R$ and suppose $0,\zeta >0$ are  fixed point of $R$ with $|R'(\zeta)|>\abs{R'(0)}>1$. 
    Then $\inf\{|\lambda-\lambda'|:\lambda\neq\lambda',\lambda,\lambda'\in \sigma(\Delta)\}=0$.
\end{theorem}
\begin{proof}
Let $\phi_\zeta$ be the inverse branch of $R$ with $\zeta$ in it's range.
Note that $\zeta$ is in the Julia set of $R$ and is an attracting fixed point of $\phi_\zeta$ and hence we may choose  $n$, and $x_1, x_2 \in \sigma(\Delta_n)$ so that $\phi_\zeta^m(x_j)\to\zeta, j=1,2$ \cite{hare2011disconnected}. Therefore,
we have,
\begin{align*}
        \abs{\phi_\zeta^m(x_1)-\phi_\zeta^m(x_2)}&=\abs{(\phi_\zeta^m(\gamma_m))'}\abs{x_1-x_2}\\
        &=\abs{\phi_\zeta'(\phi_\zeta^{m-1}(\gamma_m))}\dots\abs{\phi_\zeta'(\gamma_m)}\abs{x_1-x_2}
\end{align*}
Note that  $\phi_\zeta$ is a continuous injective map and hence monotone.
Therefore, it is no loss to assume,
$\phi_\zeta^k(x_1)\leq\phi_\zeta^k(\gamma_m)\leq\phi_\zeta^k(x_2), \forall k$.\\
 Thus, $\abs{\phi_\zeta^k(\gamma_m)-\zeta}\leq\max_{j\in\{1,2\}}\abs{\phi_\zeta^k(x_j)-\zeta}$
    Hence, by continuity of $R'$, given $\delta>0$ there is $N$ so that $m\geq k\geq N \implies \abs{R'(\zeta)}-\delta\leq\abs{R'(\phi_\zeta^k(\gamma_m)}$.
    Thus, $\frac{1}{\abs{R'(\phi_\zeta^k(\gamma_m))}}\leq\frac{1}{(R'(\zeta)-\delta)},\forall k\geq N$.
    So for $m>N$, we have,
    \begin{align*}
        \abs{\phi_\zeta^m(x_1)-\phi_\zeta^m(x_2)}&\leq \frac{1}{\abs{R'(\phi_\zeta^m(\gamma_m))}}\dots\frac{1}{\abs{R'(\phi_\zeta(\gamma_m))}}\abs{x_1-x_2}\\
        &\leq\frac{1}{(R'(\zeta)-\delta)^{m-N}}\frac{1}{R'(\phi_\zeta^N(\gamma_m))}\dots\frac{1}{R'(\phi_\zeta(\gamma_m))}\abs{x_1-x_2}.
    \end{align*}
    
    Now note that
    \[
    \abs{c_{\Delta}^{n+j+m}\phi_0^j\phi_\zeta^m(x_1)-c_{\Delta}^{n+j+m}\phi_0^j\phi_\zeta^m(x_2)}=c_{\Delta}^n\cdot c_{\Delta}^j\abs{(\phi_0^j)'(\gamma_j)}\cdot c_{\Delta}^m\abs{\phi_\zeta^m(x_1)-\phi_\zeta^m(x_2)}.
    \]
    for some $\gamma_j$ lying between $\phi_\zeta^m(x_1)$ and $\phi_\zeta^m(x_2)$. Choose $\delta$ so that $R'(0) < \abs{R'(\zeta)}-\delta$ and observe that $c_{\Delta}^m\abs{\phi_\zeta^m(x_1)-\phi_\zeta^m(x_2)}\to 0$ as $m\to \infty $.
    
     Hence by~\cite[Proposition 3.1]{shima1996eigenvalue}, it remains to show that $c_{\Delta}^j\abs{(\phi_0^j)'(\gamma_j)}$ converges to a finite number as $j\to\infty$.

    However,
    \begin{align*}
        c_{\Delta}^j\abs{(\phi_0^j)'(\gamma_j)}=& c_\Delta^j\abs{\phi_0'(\phi_0^{j-1}(\gamma_j))\dots\phi_0'(\gamma_j)}\\
        =&\frac{R'(0)}{R'(\phi_0^j(\gamma_j))}\dots\frac{R'(0)}{R'(\phi_0(\gamma_j))}\\
        =&\frac{R'(0)}{R'(\phi_0(\gamma_j))}\dots\frac{R'(0)}{R'(\phi_0^{N-1}(\gamma_j))}\prod_{k=N}^j\frac{R'(0)}{R'(\phi_0^k(\gamma_j))}.
    \end{align*}
    In addition,   
    \begin{align*}
      \prod_{k=N}^j\frac{R'(0)}{R'(\phi_0^k(\gamma_j))}=&\prod_{k=N}^j\frac{R'(0)}{R'(0)+R''(t_k)\phi_0^k(\gamma_j)}, \ \ \text{$t_k \in (0,\phi_0^k(\gamma_j))$}\\
      =&\prod_{k=N}^j\left(1-\frac{R''(t_k)\phi_0(\gamma_j)}{R'(0)+R''(t_k)\phi_0^k(\gamma_j)}\right).
    \end{align*}
    Which converges if and only if $$\sum_{k=N}^j\frac{R''(t_k)\phi_0(\gamma_j)}{R'(0)+R''(t_k)\phi_0^k(\gamma_j)}$$ converges.
    At the same time, for large $N$,  $$\sum_{k=N}^j\frac{R''(t_k)\phi_0(\gamma_j)}{R'(0)+R''(t_k)\phi_0^k(\gamma_j)}\leq C\sum_{k=N}^j\phi_0^k(\gamma_j)\leq C\sum_{k=N}^js^{k-N}\abs{\phi_0^N\gamma_j}$$ for some $C\in\mathbb{R}$ and $s \in (0,1)$. The first inequality is because $\frac{R''(t_k)}{R'(0)+R''(t_k)\phi_0^k(\gamma_j)}$ converges as $k \to \infty$ and the second inequality is because $\phi_0$ is a contraction near $0$. Since $\abs{\phi_0^N\gamma_j}$ is bounded and $\sum_{k=N}^j s^{k-N}$ is a convergent geometric series, the conclusion of the theorem follows.

\end{proof}

\begin{example}
Level - $3$ Sierpinski Gasket, $\mathcal{SG}_3$:
\ \end{example}

\begin{figure}
\label{SG3image}
\centering
\includegraphics[width=.6\textwidth]{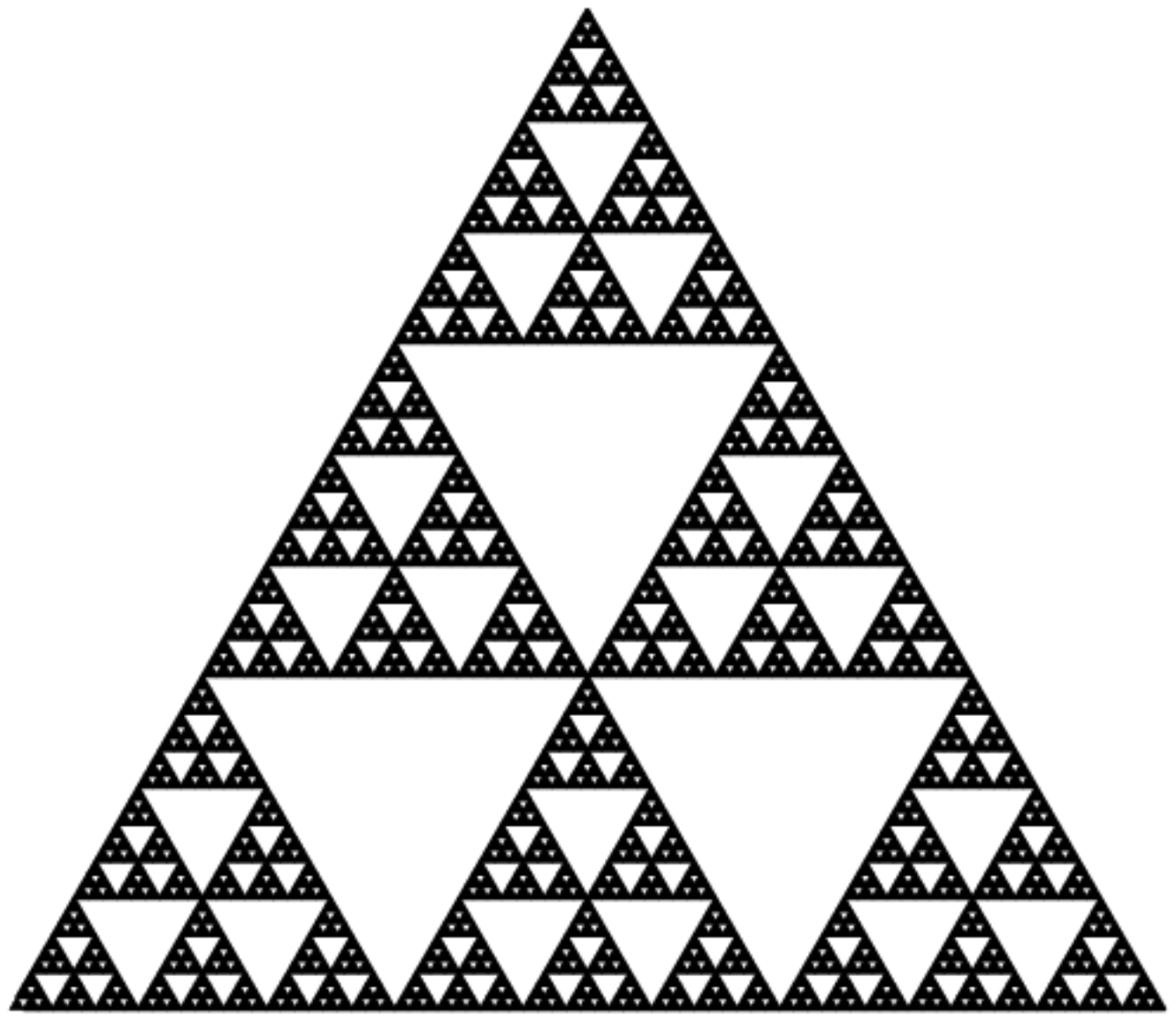}
\caption{An image of $\mathcal{SG}_3$}
\end{figure}

In \cite{bajorin2007vibration}, Bajorin et al have described the spectrum of the discrete Laplacian on $\mathcal{SG}_3$ which is given by:
\[\sigma(\Delta_n) = (R^{-n}(0))\bigcup\Bigg(R^{-(n-1)}\Bigg(\frac{3 \pm \sqrt{5}}{4}\Bigg)\Bigg)\bigcup \Bigg\{\frac{3}{2}\Bigg\}, n\geq 0\]
and whose spectral decimation function is 
\[R(x)=\frac{6x(x-1)(4x-5)(4x-3)}{6x-7}\]
with derivative,
\[R'(x) = \frac{6(288x^4-1024x^3+1290x^2-658x+105)}{(6x-7)^2}\]
Note that $0,1,\frac{3}{4},\frac{5}{4} \in R^{-n}(0) \subseteq \sigma(\Delta_n)$.
Also R has a fixed point between $\frac{5}{4}$ and $1.3$.
And it is just a matter of computation to show $R'(0)=\frac{90}{7} < 25.5428 \approx \inf_{y \in [1.2,1.5]}R'(y)$, for $y \in [1.2,1.5]$.
Hence, by {Theorem}~\ref{ZeroInf} we must have that the infimum spacing in $\sigma(\Delta)$ is zero.

\subsection{Positive Infimum Spacing}
\label{PositiveSpacing}

The proof of the next result relies on mean value theorem and induction and is crucial in proving the main Theorem~\ref{mainthm} which guarantees that the infimum spacing in the spectrum is positive under certain conditions.
\begin{lemma}
\label{main lemma}
Assume $\Delta$ admits spectral decimation with spectral decimation function $R$ and suppose $0$ is a  fixed point of $R$ with $\abs{R'(0)}>1$. 
   
    Let $D_0=\{0,x_1,\dots,x_r\}$ be a finite set with $0<x_1<\cdots<x_r$.
    Suppose, the interval $[0,x_r]$ is backward invariant under $R$ and the inverse branches,  $\phi_j's$ of $R$ are defined on all of this interval. 
    
    Assume further that
    \begin{itemize}
        \item $R^{-1}(0,x_r)\cap D_0 =\emptyset, \ \  R^{-1}\{0,x_r\}\subseteq D_0$\\
        \item $R'(0) = \max_{x\in R^{-1}[0,x_r]}|R'(x)|$
    \end{itemize}
    and set $D_n:= \cup_{m\geq 0}^n R^{-m}(D_0)$.
    Then there is a constant $C_0$, depending on $D_0$ such that,
\[
min_{x,y\in D_n, x\neq y}|x-y|\geq \frac{C_0}{(R'(0))^n}.
\]
\end{lemma}
\begin{proof}
Let $C_0 = min_k |x_k -x_{k-1}|, x_j \in D_0$.
We know $D_1= D_0 \cup R^{-1}(D_0)$.
If $x,y \in D_0, x\neq y$, the we obtain
\[
|x-y| \geq C_0 \geq \frac{C_0}{R'(0)}.
\]

If $x \in D_0$ and $y\in R^{-1}(D_0)$,
then $y\in R^{-1}[0,x_r]$. Thus $y\in \phi_k[0,x_r]$, for some k.
But $R^{-1}(0,x_r)\cap D_0 \neq \emptyset, R^{-1}\{0,x_r\} \subseteq D_0$ implies $\abs{x-y} \geq \min_{z\in \{0,x_r\}}\abs{y-\phi_k(z)}$.
For $z\in \{0,x_r\}$, we have,
$\left|\frac{R(y)-R(\phi_k(z))}{y-\phi_k(z)}\right|=\abs{R'(c)}\leq R'(0)$, for some $c\in\phi_k[0,x_r]$.
Thus, $\abs{y-\phi_k(z)}\geq \frac{\abs{R(y)-R(\phi_k(z))}}{R'(0)}\geq \frac{C_0}{R'(0)}$

For $x,y \in R^{-1}(D_0), x\neq y$:
We can assume $x,y\in \phi_k[0,x_r]$, for some $k$, since $\phi_j\{0,x_r\}\subseteq D_0$ for all $j$. 
$R$ bijection on this interval implies $R(x) \neq R(y)$.
We obtain 
\begin{align*}
\frac{|R(x)-R(y)|}{|x-y|} &= R'(c)\leq R'(0) & \text{for some $c\in \phi_k[0,x_r]$}\\
\end{align*}
Thus,
\begin{align*}
 |x-y|&\geq \frac{|R(x)-R(y)|}{R'(0)}\geq \frac{C_0}{R'(0)}
 &\text{since $R(x), R(y) \in D_0$}.
\end{align*}
The general case follows by induction and this completes the proof.
\end{proof}
 
\begin{theorem}
\label{mainthm}
    Suppose there is $D_0$, so that $\sigma(\Delta_n) \subseteq D_n$ for all $n$ and $R,D_n$ satisfy the conditions of {Lemma}~\ref{main lemma}.
    Then,
    \[
    \inf\{|\lambda-\lambda'|:\lambda\neq\lambda' \text{ are Dirichlet eigenvalues of } \Delta\}>0
    \]
    \[
    \inf\{|\lambda-\lambda'|:\lambda\neq\lambda' \text{ are Neumann eigenvalues of } \Delta\}>0.
    \]
\end{theorem}
\begin{proof}
It follows from {Lemma}~\ref{main lemma} that, the spacing in  $\sigma(c_\Delta^n\Delta_n)$ are bounded below uniformly in $n$, where $c_\Delta = R'(0)$. The conclusion follows.
\end{proof}

 Though the following examples are very easy, they may be useful for understanding the results obtained in this section and how they apply to other complicated scenarios.
\begin{example}($\mathcal{SG}$.)
\end{example}

Alonso Ruiz shows in \cite[Theorem 3]{alonso2022minimal} that,
\[\min\{|\lambda - \lambda'|: \lambda \neq \lambda', \text{Dirichlet eigenvalues of } \Delta\} = \lambda_0^{(5)} - \lambda_0^{(2)} >0\]
Thus the infimum spacing of eigenvalues coincides with the spectral gap which is positive.

The same paper, \cite[Theorem 3]{alonso2022minimal}, gives that
\[\min\{|\lambda - \lambda'|: \lambda \neq \lambda', \text{Neumann eigenvalues of } \Delta\} = \lambda_0^{(6)}  > 0\]
and we see the infimum spacing is again positive.

\begin{example}($I$)
\end{example}
Since we know the spectrum of the Laplacian on $I$, It is just a matter of easy computation to show that the infimum spacing is again the spectral gap.\\
Note that the Dirichlet spectrum is given by $\{\pi^2k^2\}_{k=1}^\infty$.
Hence spacing between subsequent eigenvalues is given by 
\[\pi^2(k+1)^2 -\pi^2k^2 
=\pi^2((k+1)^2-k^2)
=\pi^2(2k+1)\]
which is increasing in $k$, hence the infimum spacing is again the spectral gap which occurs when $k=1$.
Thus
\[\min\{|\lambda - \lambda'|: \text{ Dirichlet eigenvalues of } \Delta\} = 3\pi^2 > 0\].

Recall the Neumann spectrum is given by $\{\pi^2k^2\}_{k=0}^\infty$.
Thus the spacing between subsequent eigenvalues is again given by:
$\pi^2(2k+1)$, which coincides with the Dirichlet case, with minimum occurring when $k=0$.
Thus 
\[\min \{|\lambda - \lambda'|: \text{Neumann eigenvalues of } \Delta\}= \pi^2 >0. \]

\begin{remark}
    Typically, for an appropriate choice of $D_0$, we pick a large $n$, and set,
    $D_0=R^{n}\sigma(\Delta_n)$.
\end{remark}

\begin{appendices}
   \section{Wielandt's Inequalities} 
   In the {\it theorem} below, $\lambda^{\downarrow}(T)$ represents the spectrum of the operator $T$ in decreasing order and $\lambda_j^{\downarrow}(T)$ is the $j'th$ element of this set.
   \begin{theorem}{(Wielandt's Inequalities)}\cite{koltchinskii2000random}
   \label{wielant_inequalites}
    Let A be a symmetric operator on $\mathbb{R}^n$ such that $\lambda_d^{\downarrow}(A)-\lambda_{d+1}^{\downarrow}(A)>0$, for some $1\leq d <n$. Let $P^d$ denote the orthogonal projector of $\mathbb{R}^n$ onto the subspace generated by the eigenvectors corresponding to the largest $d$ eigenvalues of $A$, and let $P_d$ denote the orthogonal projector onto the subspace generated by the remaining eigenvectors of $A$. If $B$ is a symmetric operator such that $P_dBP_d=0$ and $P^dBP^d=0$, then 
    \begin{equation}
        0\leq \lambda_j^{\downarrow}(A+B)-\lambda_j^{\downarrow}(A) \leq \frac{\norm{B}^2}{\lambda_j^{\downarrow}(A)-\lambda_{d+1}^{\downarrow}(A)}, j=1,\dots,d,
    \end{equation}
    and 
    \begin{equation}
        0 \leq \lambda_j^{\downarrow}(A)-\lambda_j^{\downarrow}(A+B)\leq \frac{\norm{B}^2}{\lambda_d^{\downarrow}(A)-\lambda_j^{\downarrow}(A)}, j=d+1,\dots,n.
    \end{equation}
   \end{theorem}
\end{appendices}

\subsection*{Acknowledgments}  

This research was partially  supported by NSF grant  DMS-2349433. \\
I am deeply grateful to Alexander Teplyaev for his invaluable guidance while working on this project.
I also acknowledge the interesting conversation on some aspects of this work with Christopher Hayes and Kasso Okoudjou.

\bibliographystyle{amsalpha}
\bibliography{paperref}

\end{document}